\documentclass[12pt]{amsart}
\usepackage{graphicx}
\usepackage{a4wide}
\usepackage{amssymb}
\usepackage{amsthm}
\usepackage{amsmath}
\usepackage{amscd} \usepackage{verbatim}
\usepackage[all]{xy}
\textheight8.5in \textwidth6.7in \numberwithin{equation}{section}
\theoremstyle{plain}
\newtheorem{theorem}{Theorem}[section]

\newtheorem{lemma}[theorem]{Lemma}

\theoremstyle{definition}
\newtheorem{definition}[theorem]{Definition}

\theoremstyle{remark}
\newtheorem*{remark}{Remark}

\newcommand{\maxp}{\text {\rm maxp}}

\newcommand{\N}{\mathbb{N}}

\newcommand{\D}{\mathbf{D}}

\newcommand{\ga}{\gamma}

\begin{document}
\title[Multiplicative properties of $k$-regular partitions]
{Multiplicative properties \\ of the number of $k$-regular partitions}
\author{Olivia Beckwith and Christine Bessenrodt}
\address{Department of Mathematics and Computer Science,
Emory University, Atlanta, Georgia 30322}
\email{obeckwith@gmail.com}
\address{Faculty of Mathematics and Physics,
Leibniz University Hannover,
Welfengarten 1,
D-30167 Hannover,
Germany}
\email{bessen@math.uni-hannover.de}
\thanks{}
\begin{abstract}
In a previous paper of the second author with K.\ Ono,
surprising multiplicative properties of the partition
function were presented.
Here, we deal with $k$-regular partitions.
Extending the generating function for $k$-regular partitions
multiplicatively to a function on $k$-regular partitions,
we show that it takes its maximum
at an explicitly described small set of partitions,
and can thus easily be computed.
The basis for this is an extension of a classical
result of Lehmer, from which an inequality for the generating
function for $k$-regular partitions is deduced
which seems not to have been noticed before.
\end{abstract}
\date{September 9, 2014}
\maketitle

\section{Introduction and statement of results}\label{sec:intro}

A partition of a natural number $n$
is a finite weakly decreasing sequence of positive integers
that sums to~$n$.
For $k\in \N$, $k>1$,
we consider the generating function $p_k(n)$
that enumerates $k$-regular partitions of $n$,
i.e., it counts partitions of $n$
where no part
is divisible by~$k$.
These generating functions arise
in many different contexts, in particular in connection
with the representation theory of the symmetric groups, Hecke algebras,
and related groups and algebras;
for a long time, this has been studied
both in combinatorics and number theory.

For the classical (unrestricted) partition function $p(n)$,
explicit formulae are known due to
the work of Hardy, Ramanujan and Rademacher, and more recent work
of Bruinier and Ono~\cite{BruinierOno}.
Based on a result due to Lehmer, the following inequality was
shown in a recent article by the second author and Ono~\cite{BesOno}:
\medskip

\centerline{
For any integers $a,b$ such that  $a,b >1$ and $a+b>9$, we have
$p(a)p(b)>p(a+b)$.}
\medskip

Also the cases of equality were determined
in \cite{BesOno}. The inequality above
was then used to study an ``extended partition function",
given by defining for a partition $\mu=(\mu_1,\mu_2,\ldots)$:
$$
p(\mu)=\prod_{j\ge 1} p(\mu_j).
$$
With $P(n)$ denoting the set of all partitions of~$n$,
the maximum
$$\maxp(n)=\max(p(\mu) \mid \mu \in P(n))$$
was determined explicitly in \cite{BesOno};
we recall this below in Theorem~\ref{thm:BO-max}.

Our aim is to prove a corresponding result for an
extension of the generating function $p_k(n)$ to
a function on the set $P_k(n)$ of
all $k$-regular partitions of $n$,
defined for $\mu=(\mu_1,\mu_2,\ldots)\in P_k(n)$ by:
$$
p_k(\mu)=\prod_{j\ge 1} p_k(\mu_j).
$$
We then determine on which partitions the maximum
$$\maxp_k(n)=\max(p_k(\mu) \mid \mu \in P_k(n))$$
is attained, and we use this to give an explicit formula for the
maximum.

By Theorem~\ref{thm:BO-max},  for $k>6$
nothing new happens, as all the partitions providing the maximal values
$\maxp(n)$ are already $k$-regular; hence we may restrict our considerations to
the cases where $2\le k\le 6$.
For this case,
we first show  in Theorem~\ref{thm:reg-inequality}
that $p_k(n)$ satisfies a similar inequality
as the one given for $p(n)$ above,
where again we specify the corresponding bounds explicitly.

For the maximum problem,
we find that the behavior is quite similar to the one
observed in~\cite{BesOno}, though we lose uniqueness for small~$k$;
see Theorem~\ref{thm:max-regular} for the detailed results.

\section{An analytic result on the generating function for $k$-regular partitions}

The main result of this section is
the following analytic inequality for the generating function~$p_k(n)$.
As mentioned above,
Theorem \ref{thm:reg-inequality} is the analogue of a result for the
ordinary partition function $p(n)$ in recent work by the second author and Ono~\cite{BesOno}.

\begin{theorem}\label{thm:reg-inequality}
For $k\in \N$, $2\leq k\leq 6$, we define parameters $n_k,m_k$ by the following table:
\begin{center}

\begin{tabular}{ | c | c|  c | c | c | c |}
\hline
$k$             & 2 & 3 &4  & 5 & 6\\
\hline
$n_k$       	&3 & 2 & 2 & 2  &2 \\

$m_k$           & 22& 17 & 9 & 9 & 9 \\
\hline

\end{tabular}
\end{center}
Then for any $a, b\in \N$ with $a,b\geq n_k$ and $a+b\geq m_k$
we have
   $$
       p_k(a) p_k(b) > p_k(a+b).
   $$
Furthermore, all the pairs $(a,b)$ with $2 \le a \le b$ for which this inequality
fails are given in the table below.
$$
\begin{array}{|l|l|l|}
\hline
k & (a,b) \text{ with } p_k(a) p_k(b) = p_k(a+b) & (a,b) \text{ with } p_k(a) p_k(b) < p_k(a+b)\\
\hline
2&  (3, 3), (3, 5), (3, 6), (3, 7), (3, 8), (4, 15),
& (2, *), (3, 4), (4, 4), (4, 5),(4, 6), (4, 7),
\\
  &   (4, 16), (4, 17),(5, 6), (5, 7), (5, 8)
  &  (4, 8), (4, 9), (4, 10),(4, 11), (4, 12), \\
  &
  & (4, 13), (4, 14),  (5, 5)   \\
\hline
3&  (2, 2), (2, 3), (3, 3), (3, 4), (3, 5), (3, 6),
& (3, 11), (3, 13)\\
  & (3, 7), (3, 8), (3, 9),(3, 10)  &  \\
\hline
4&  (2, 2), (2, 3),  (2, 5), (3, 3)  & (2, 4), (3, 5)  \\
\hline
5&   (2,3), (2,4)  & (2,2), (2,5), (3, 3), (3, 5)\\
\hline
6&   (2, 4), (2, 5), (2, 6) & (2, 2), (2, 3),  (3, 3)\\
\hline
\end{array}
$$
\end{theorem}

\bigskip

The main tool for deriving Theorem~\ref{thm:reg-inequality} is an analogue of a classical result of D.~H.\ Lehmer~\cite{Lehmer}.
To prove Theorem~\ref{thm:reg-inequality} we need precise approximations for $p_k(n)$ which have explicitly bounded error terms.
We will use work of Hagis \cite{Hagis} to obtain sufficient approximations in Theorem \ref{thm:Lehmer}.

\subsection{Preliminaries}
Hagis \cite{Hagis} proved an explicit formula for $p_k(n)$ that is analogous to Rade\-macher's formula for $p(n)$.
Before describing his theorem, we introduce several necessary quantities, most importantly the Kloosterman-type
sums $A(m,t,n,s,D)$ and the expressions $L(m,t,n,s,D)$.

Let $D$ divide $t+1$, let $J=J(t,D) := \frac{(t/D) - D}{24D}$, and let $a=a(t) := \frac{t}{24}$.
Let $I_1$ be the order one modified Bessel function of the first kind, and let $L(m,t,n,s,D)$ be given by
\begin{equation}\label{eq:Lsum}
L(m,t,n,s,D) := D^{\frac{3}{2}} m^{-1} \left(\frac{J-s}{n+a} \right)^{\frac{1}{2}} I_1 \left (4 \pi D m^{-1} \left( \frac{(J-s)(n+a)}{(t+1)} \right)^{\frac{1}{2}} \right).
\end{equation}

Several definitions are needed to define the modified Kloosterman sums $A(t,m,n,s,D)$.
First $g=g(m)$ is defined to be $\gcd(3,m)$ when $m$ is odd, and $8\gcd(3,m)$ when $m$ is even.
We define $M=M(m,D):= \frac{m}{D}$.
Additionally, we define $f=f(m)  := \frac{24}{g}$,
and define $r=r(m)$ to be any integer such that $fr \equiv 1 \pmod{gm}$.
Further, $G$ is defined to be analogous to $g$,
in that $G=G(m,D) := \gcd(3,M)$ when $M$ is odd and $G := 8\gcd(3,M)$ when $M$ is even.
Then we also let $B=B(m,D) := \frac{g}{G}$, and we define $A$ to be any integer such
that $AB \equiv 1 \pmod{GM}$. We also let $T=T (t,D):= \frac{t+1}{D}$, and choose $T'=T'(t,D)$
to satisfy $T T' \equiv 1 \pmod{GM}$. More importantly:
\begin{equation*}
U=U(t,m,D) := 1 - AB(t+1), V=V(t,m,D):=ABT'D - 1.
\end{equation*}

Hagis defines special roots of unity, $w(h,t,m,D)$, which satisfy the following:
\begin{equation*}
w (h,t,m,D) = C(h,t,m,D) \exp (2 \pi i(rUh + rV h')/gm).
\end{equation*}
The $C(h,t,m,D)$ satisfy $|C(h,t,m,D)| = 1$, and are independent of $h$ if $m$ is odd, or if $m$ is even and we restrict to $h \equiv d \pmod{8}$ for some odd $d$. In what follows we will not explicitly use the definitions of $C(h,t,m,D)$.

Then we define $A(m,t,n,s,D)$ to be the Kloosterman sum with multiplier system given by
\begin{equation}
A(m,t,n,s,D) = \sum_{\substack{h \pmod{m}, \\ \gcd(h,m)=1} } w(h,t,m,D) \exp(- 2 \pi i( nh  - DT' sh')/ m),
\end{equation}
where $h h' \equiv 1 \pmod{gm}$.

Let $p'(s)$ be the number of partitions of $s$ into an even number of distinct parts
minus the number of partitions of $s$ into an odd number of distinct parts;
by Euler's pentagonal number theorem, $p'(s)$ is $\pm 1$ if $s$ is a pentagonal number, and 0 otherwise.
Recall Glaisher's partition identity saying that
the number $p_k(n)$ of $k$-regular partitions of $n$
is equal to the number of partitions of $n$ where no part
has a multiplicity $\geq k$.
Using the previous notation, Hagis proved the following for the numbers $p_k(n)$
in \cite[Theorem 3]{Hagis}.
\begin{theorem}\label{thm:Hagis}
For all $k \ge 2$, the number of $k$-regular partitions of $n\in \N$ is given by
\begin{equation}
p_k(n) = \frac{2 \pi}{k} \sum_{\substack{D | k \\  D < k^{\frac{1}{2}}}}
\; \sum_{\substack{m\\ \gcd(k,m) = D}}^{\infty} \mbox{ }\sum_{s < J(k,D)} p'(s) A(m, k-1,n,s,D) L(m,k-1,n,s,D).
\end{equation}
\end{theorem}
For $2 \le k \le 6$, in the summations above, we only have $s = 0$ and $D \le 2$.
Thus, the formulae needed for Theorem \ref{thm:reg-inequality}
consist of one or two of the inner sums in Theorem~\ref{thm:Hagis}.

\subsection{Estimates in the theorem of Hagis}
In this section, we obtain an asymptotic for $p_k(n)$ with an explicitly bounded error term.

Let $\alpha_k$ be defined as follows:

\begin{equation}\label{eq:alpha}
\alpha_k := \left\{
  \begin{array}{lr}
   1.8  & \mbox{ if } k=2 \\
   9.84 & \mbox{ if } k=3 \\
   1.8 \cdot 3^{\frac{1}{2}}  & \mbox{ if } k=4 \\
   14.37 & \mbox{ if } k=5 \\
   1.23 \cdot 5^{\frac{1}{2}}  & \mbox{ if } k=6
  \end{array}
\right .
\end{equation}
We also let $\alpha_6 ':= 19.68$.
\begin{theorem}\label{thm:Lehmer}
For $n\in \N$, let $\mu=\mu(n,k) := \frac{\pi ((k-1)^2 + 24n(k-1))^{\frac{1}{2}}}{6 k^{\frac{1}{2}}}$.
\begin{enumerate}
 \item  
 For $2\leq k\leq 5$ we have:
\begin{equation*}
p_k (n) = \frac{2 \pi}{k} \left( \frac{k-1}{k-1 + 24n} \right)^{\frac{1}{2}} I_1(\mu) + E_k(n)
\end{equation*}
where
\begin{equation*}
|E_k(n)| <  \frac{\alpha_k \pi}{k} \left( \frac{k-1}{(k-1)+ 24 n} \right)^{\frac{1}{2}} \frac{1}{\mu} e^{\mu} (1 + 5 \mu^2 e^{- \mu}).
\end{equation*}

 \item 
 For $k=6$ we have:
\begin{equation*}
p_6(n) = \frac{\pi}{3} \left(\frac{5}{24n + 5} \right)^{\frac{1}{2}} I_1 (\mu) + E_6(n)
\end{equation*}
where
\begin{equation*}
\begin{split}
| E_6(n) | &< \frac{\pi}{3} \left(\frac{5}{24n + 5} \right)^{\frac{1}{2}} \frac{\alpha_6}{2} \frac{e^{\mu}}{\mu} \left( 1 + \delta (n) \right)
+ \frac{\pi}{3} \left(\frac{1}{24n + 5}\right)^{\frac{1}{2}} I_1 \left( \frac{ \mu}{10^{\frac{1}{2}}} \right) .
\end{split}
\end{equation*}
\end{enumerate}

where $\delta(n)  := 5 \mu^2 e^{-\mu} + \frac{2^{\frac{4}{3}} \alpha_6'}{\alpha_6} e^{\mu \left( \frac{1}{\sqrt{10}} -1 \right)} \left( 1 + e^{- \frac{\mu}{\sqrt{10}}} \frac{\mu^2}{2} \right) $.
\end{theorem}

\begin{remark}
Theorem \ref{thm:Lehmer} is analogous to \cite[(4.14)]{Lehmer} in the case of $p(n)$.
\end{remark}
To prove this theorem, we need some preparations.
The first is a bound on the divisor counting function $d(n)$.
\begin{lemma}\label{thm:divisorbound}
Let $d(n)$ denote the number of positive divisors of a positive integer $n$.
\begin{enumerate}
\item 
For all $n$, $d(n) \le 3.57 n^{\frac{1}{3}}$.

\item 
If $n$ is odd, then $d(n) \le 1.8 n^{\frac{1}{3}}$.

\item 
If $\gcd(n,3) =1$, then $d(n) \le 2.46 n^{\frac{1}{3}}$.

\item 
If $\gcd(n,5) = 1$, then $d(n) \le 3.05 n^{\frac{1}{3}}$.

\item 
If $\gcd(n,6)=1$, then $d(n) \le 1.23 n^{\frac{1}{3}}$.
\end{enumerate}
\end{lemma}
\begin{remark}
Actually, it is known that $d(n) = O(n^{\epsilon})$ for any $\epsilon >0$ (see ~\cite{Wigert}). However,
to prove our main theorem it is necessary that we have exact constants for the bounds. We chose these exponents and constants carefully to ease the calculations in the proof of Theorem \ref{thm:reg-inequality}.
\end{remark}
\begin{proof}
Let $n = \prod_{i =1}^M  {p_i}^{a_i}$, where each $p_i$ is prime. Then $d(n) = \prod_{i = 1}^M (1 + a_i)$. We follow the classical method in \cite{Hardy} of bounding $\prod_{i = 1}^M \frac{a_i+ 1}{p_i^{\frac{a_i}{3}}}$. For $p_i \ge 11$, we have $\frac{a_i + 1}{p_i^{\frac{ a_i}{3}}} \le 1$ for $a_i \ge 1$. For the remaining $p_i$, the quantity $\frac{a_i + 1}{p_i^{\frac{ a_i}{3}}}$ is maximized when $a_i$ is equal to $ 3,2,1$ and $1$ for $p_i$ equal to $2,3,5$ and $7$, respectively. The lemma follows by maximizing $\prod_{i = 1}^M \frac{a_i+ 1}{p_i^{\frac{a_i}{3}}}$ over $n$ which respect each of the given divisibility constraints.
\end{proof}
The next lemma is a bound on $A(m,k-1,n,0,D)$, which is related to the classical Kloosterman sums defined below;
it is a slight modification  of \cite[Theorem 12]{Lehmer}.
\begin{definition}
Let $a,b,m\in \N$. The Kloosterman sum $S(a,b,m)$ is defined by
\begin{equation*}
S(a,b,m) : = \sum_{\substack{1 \le h \le m-1 \\ \gcd(h,m) = 1}} e^{2 \pi i (a h + bh')/m},
\end{equation*}
where $h'$ is the multiplicative inverse of $h$ modulo $m$.
\end{definition}

 Weil proved the following bound (see \cite[Theorem 4.5]{Iwaniec}):
\begin{theorem}\label{thm:weilbound}
Let $a,b,m\in \N$.
\begin{equation*}
|S(a,b,m)| \le d(m) m^{\frac{1}{2}} \gcd(a,b,m)^{\frac{1}{2}}.
\end{equation*}
\end{theorem}

We will use this bound in the following lemma.

\begin{lemma}\label{thm:Asums}
\begin{enumerate}
\item 
For $2 \le k \le 6$, and for all $n,m \ge 1$ with $\gcd(k,m) = 1$, we have
\begin{equation*}
|A(m,k-1,n,0,1)| < \alpha_k m^{\frac{5}{6}}.
\end{equation*}

\item 
For all $n,m \ge 1$ with $\gcd(6,m) = 2$, we have
\begin{equation*}
|A (m,5,n,0,2) | < \alpha_6' m^{\frac{5}{6}}.
\end{equation*}
\end{enumerate}
\end{lemma}
\begin{proof} We will follow Hagis' argument in \cite[Theorem 2]{Hagis}.
Our strategy is to rewrite $A(m,k-1,n,0,D)$ as a sum of ordinary Kloosterman sums and apply
Theorem~\ref{thm:weilbound}.

In order to bound the ordinary Kloosterman sums, we will need to be able to bound certain greatest common divisors.
We use the notation introduced at the beginning of the section,
and we begin by stating a series of bounds for
$\gcd(Ur - gn, rV, gm)$ and $\gcd(Ur - gn, rV + \frac{wgm}{8}, gm)$ which depend on $k$ and $D$.
These are straightforward to verify from their definitions.

For $D = 1$, $ 2 \le k \le 6$ we have $\gcd(r, gm) = 1 $ and $\gcd(k,gm) = 1$, thus \\
$\gcd(rV, gm) = \gcd(kV,gm)$. Then since $kV = k(T' - 1)  \equiv 1-k \pmod{gm}$, we have
\[ \gcd(rV, gm)  = (1-k, gm) \le k-1. \]

Let $k= 3,5$, let $D = 1$, and let $m$ be even. Note that $\gcd(r,g) = 1$ and $U= k-1$, which implies
$\gcd(Ur - gn , g) = \gcd(k-1, g)$. Also for $1 \le w \le 8$, we have
\[\gcd(rV + \frac{wgm}{8} , m) = \gcd(V, m) = \gcd( 1-k,m). \]
Therefore $\gcd(Ur -gn, rV + \frac{wgm}{8} , gm)$ divides $(k-1)^2$, so it must be $1,2,4,8$, or 16.
However, the highest power of $2$ that $Ur - gn$ can be divisible by is $k-1$, because $g$ is divisible by $8$, and $r$ is odd, and $Ur - gn = r (1-k) - gm$. Thus we have:
\begin{equation*}
\gcd(Ur - gn, Vr + \frac{wgm}{8}, gm) \le k-1.
\end{equation*}

For the last bound, we let $k=6$ and $D=2$. Then we have $g= 8$, $T = 3$, $M = \frac{m}{2}$, and $\gcd(6,m) = 2$.
So $\gcd(rV + \frac{wgm}{8}, m) = \gcd(V, m) = \gcd(2 ABT' - 1,m)$. Now we have \\ $2 AB \equiv 2 \pmod{m}$ and $6T' \equiv 2 \pmod{m}$, thus \[ \gcd(V,m) = \gcd(2T' - 1,m) = \gcd(3(2 T' - 1),m )  = 1. \] Therefore we have $\gcd( rU - gn, rV + \frac{wgm}{8},gm)\le g = 8.$

To use these bounds, we rewrite $A(m,k-1,n,0,D)$ as a sum over a reduced residue class modulo $gm$:
\begin{equation*}
A(m,k-1,n,0,D) =
\frac{1}{g} \sum_{\substack{h \mod{m} \\ \gcd(h,m) = 1}} C(h,k-1,m,D) \exp{(2 \pi i ((Ur-gn)h + rVh')/gm))}.
\end{equation*}
For odd $m$, $C(h,k-1,m,D)$ does not depend on $h$. Therefore we have
\begin{equation*}
A(m,k-1,n,0,1) = C(1,k-1,m,1)  \frac{1}{g} \sum_{\substack{h \mod{m} \\ \gcd(h,m) = 1}} \exp(2 \pi i  ((rU-gn)h + rVh')/gm) .
\end{equation*}
The sum on the right is an ordinary Kloosterman sum, so by Theorem \ref{thm:weilbound} we have, for all odd~$m$:
\begin{equation*}
|A(m,k-1,n,0,1)|  = |S(Ur - gn, rV, gm)| \le \frac{1}{g} d(gm) \gcd(Ur - gn, rV, gm)^{\frac{1}{2}} (gm)^{\frac{1}{2}}.
\end{equation*}
Then by Lemma \ref{thm:divisorbound} and the bounds at the beginning of the proof, it follows that for all $m$ such that $2 \nmid m$ and $\gcd(k,m) = 1$, we have:
\begin{equation*}
|A(m,k-1,n,0,1)| \le (k-1)^{\frac{1}{2}} \cdot 1.8 \cdot m^{\frac{5}{6}}.
\end{equation*}
This proves the lemma for $k=2,4$, and for $k=3,5$ in the case of $m$ being odd.
Similarly, for $k=6$, Lemma \ref{thm:divisorbound}, we have:
\begin{equation*}
|A(m,5,n,0,1)| \le (k-1)^{\frac{1}{2}} \cdot 1.23  \cdot m^{\frac{5}{6}}.
\end{equation*}
For $k=6$, $D=1$, the proof is complete.

If $m$ is even, we write
\begin{align*}
A(m,k-1,n, 0, D) &= A_1(m,k-1,n,0,D) + A_3(m,k-1,n,0,D) \\
&\qquad + A_5(m,k-1,n,0,D) + A_7(m,k-1,n,0,D),
\end{align*}
where
\begin{equation*}
A_d (m,k-1,n,0,D) = \frac{1}{g} \sum_{\substack{h \mod{gm}, \\ h \equiv d \mod{8}, \\ \gcd(h,m) = 1}} C(h,k-1,m,D) \exp{(2 \pi i  ((rU-gn)h +rVh')/gm)}.
\end{equation*}
Over each $d$, the coefficient $C(h,k-1,m,D)$ does not depend on $h$, so
\begin{equation*}
A_d (m,k-1,n,0,D) = C(d,k-1,m,D) \frac{1}{g} \sum_{\substack{h \pmod{gm},  \\ h \equiv d \pmod{8},\\ \gcd(h,m) = 1 }} \exp{(2 \pi i  ((rU-gn)h +rVh')/gm)}.
\end{equation*}
By the formula on page 266 of \cite{Salie}, for $d d' \equiv 1 \pmod{8}$, we have:
\begin{equation*}
A_d (k-1,m,n,0,D) = \frac{1}{8g} C(d,k-1,m,D) \sum_{w = 1}^8 e^{2\pi i \frac{d' w}{8}}  S(Ur-gn, Vr + \frac{w gm}{8}; gm).
\end{equation*}
By Theorem \ref{thm:weilbound},
\begin{equation*}
A_d (m,k-1,n,0,D) = \frac{1}{8g} C(d,k-1,m,D) \sum_{w = 1}^8 e^{- \frac{2 \pi i}{8} d'w} \gcd(Ur-gn, Vr + \frac{w gm}{8},gm)^{\frac{1}{2}} d(gm) (gm)^{\frac{1}{2}}.
\end{equation*}
For $k=3$, by the bounds at the beginning of the proof we have:
\begin{equation*}
A_d(m,2,n,0,1) \le 8 \cdot \frac{1}{8g} \cdot 2^{\frac{1}{2}}\cdot 2.46 (gm)^{\frac{1}{3}} \cdot (gm)^{\frac{1}{2}} \le 2.46 m^{\frac{5}{6}}.
\end{equation*}
Similarly for $k=5$, if $3 | m$, by our previous bounds we have:
\begin{equation*}
A_d (m,4,n,0,1) \le  8 \cdot \frac{1}{8 \cdot 24} \cdot 4^{\frac{1}{2}} \cdot 3.05 \cdot (24m)^{\frac{1}{3}} \cdot (24m)^{\frac{1}{2}} \le 3.592 m^{\frac{5}{6}}.
\end{equation*}
If $3 \nmid m$, then we have:
\begin{equation*}
|A_d (m,4,n,0,1)| \le  8 \cdot \frac{1}{8\cdot 8} \cdot 4^{\frac{1}{2}} \cdot 2.46 \cdot (8m)^{\frac{1}{3}} \cdot (8m)^{\frac{1}{2}} \le 3.48 m^{\frac{5}{6}} .
\end{equation*}
We note that $|A(m,k-1,n,0,D)| \le 4 |A_d (m,k-1,n,0,D)|$. Comparing these bounds to the bounds in the odd $m$ case, we conclude that for $k = 3,5$, the desired bound holds whenever $\gcd(m,k) = 1$.

For $\gcd(6,m) = 2$, we have:
\begin{equation*}
|A(m,5,n,0,2) | \le 4 |A_d (m,5,n,0,2)| \le 4 \cdot (8\cdot 8^{\frac{1}{2}} \cdot \frac{1}{8g}  \cdot 2.46 (gm)^{\frac{1}{3}} \cdot (gm)^{\frac{1}{2}}) \le 19.6 m^{\frac{5}{6}}.
\end{equation*}
This completes the proof.
\end{proof}

\bigskip

Now we come to the {\bf proof of Theorem \ref{thm:Lehmer}}.
For $2 \le k \le 5$, Theorem \ref{thm:Hagis} says
\begin{equation}\label{eq:hagis1}
p_k(n) = \frac{2 \pi}{k} \sum_{\substack{m=1\\ \gcd(k,m) = 1}}^{\infty} m^{-1} \left(\frac{k-1}{(k-1)+ 24 n} \right)^{\frac{1}{2}} A(m,k-1,n,0,1) I_1 \left( \frac{\mu}{m} \right),
\end{equation}
and for $k=6$, Theorem \ref{thm:Hagis} says
\begin{equation}\label{eq:hagis2}
\begin{split}
p_6(n) = &\frac{\pi}{3} \frac{5^{1/2}}{(5 + 24n)^{\frac{1}{2}}} \sum_{m= 1}^{\infty} \frac{1}{m} A(m,5,n,0,1) I_1 \left(\frac{\mu}{m} \right) \\
&+ \frac{\pi}{3} \frac{1}{(5 + 24n)^{\frac{1}{2}}} \sum_{(3,a) = 1}^{\infty} \frac{1}{a} A(2a,5,n,0,2) I_1 \left( \frac{ \mu}{10^{\frac{1}{2}} a} \right).
\end{split}
\end{equation}

Let $\alpha = \frac{1}{6}$. Our proof works by bounding the sums in~(\ref{eq:hagis1}) and~(\ref{eq:hagis2}).
We have, for any $\nu \neq 0$,
\begin{align*}
|\sum_{m= N+1}^{\infty} m^{-1} A(m,k-1,n,0,1) I_1 \left(\frac{\nu}{m} \right)|  &\le \sum_{m= N+1}^{\infty} \alpha_k m^{- \alpha} \sum_{j=0}^{\infty} \frac{ (\frac{\nu}{2m})^{2j+1}}{j! (j+1)!} \\
&< \alpha_k \int_N^{\infty} x^{-\alpha} \sum_{j=0}^{\infty} \frac{ (\frac{\nu}{2x})^{2j+1}}{j! (j+1)!} \D{x}.
\end{align*}
We substitute $t = \frac{\nu}{2x}$.
\begin{align*}
|\sum_{m= N+1}^{\infty}m^{-1} A(m,k-1,n,0,1) I_1 \left(\frac{\nu}{m} \right)| &\le \alpha_k  \int_0^{\frac{\nu}{2N}} (\frac{\nu}{2t})^{-\alpha} \sum_{j=0}^{\infty} \frac{t^{2j+1}}{j! (j+1)!} \frac{\nu}{2t^2} \D{t} \\
&= \alpha_k  (\frac{\nu}{2})^{1 - \alpha} \int_0^{\frac{\nu}{2N}} \sum_{j=0}^{\infty} \frac{(t^{2j -1 +\alpha})}{j! (j+1)!} \D{t} \\
&\le \alpha_k \ (\frac{\nu}{2})^{1 - \alpha}  \sum_{j=0}^{\infty} \frac{(\frac{\nu}{2N})^{2j+\alpha}}{j! (j+1)! (2j + \alpha)} \\
&\le \alpha_k  (\frac{\nu}{2})^{1- \alpha}  \left(\frac{ (\frac{\nu}{N})^{\alpha}}{2 \alpha} + \sum_{j=2}^{\infty} \frac{((\frac{\nu }{N})^{2j-2 + \alpha})}{(2j)!}  \right)2^{1 - \alpha}\\
&\le \alpha_k N^{2 + \alpha} \frac{1}{\nu} \left( \cosh(\nu/N) -1 + \frac{5}{2} \left(\frac{\nu}{N} \right)^2 \right). \\
\end{align*}

To bound $\sum_{a = N+1}^{\infty} (2a)^{-1} A(2a,5,n,0,2) I(\frac{\nu}{2a})$, we replace $\alpha_6$ with $\alpha_6'$ in the previous argument.
To complete the proof, we let $N = 1$, and apply the above inequality to the sums in Theorem~\ref{thm:Hagis}, where $\nu=\mu$ for $2 \le k \le 5$,
and for $k = 6$, $\nu$ is set to be $\mu$ and $\frac{\mu}{\sqrt{10}}$ in the first and second sum, respectively.
\qed

\subsection{Proof of Theorem \ref{thm:reg-inequality}}\label{sec:proof-reg-inequality}
Our proof is analogous to the proof of  \cite[Theorem 2.1]{BesOno}.

By well known properties of Bessel functions, such as the bounds in
(9.8.4) of \cite{Abram}, for $x \ge 37.5$ the modified Bessel function $I_1(x)$
is bounded by
\begin{equation*}
N \le x^{\frac{1}{2}} e^{-x} I_1(x) \le M
\end{equation*}
where $N = 0.394$, $M = 0.399$.

First, let $2 \le k \le 5$, and let $\beta := \frac{\alpha_k}{2}$. Then by Theorem \ref{thm:Lehmer}, for $n \ge 450$ we have:
\begin{align*}
\frac{2 \pi}{k} \left( \frac{k-1}{k-1 + 24n} \right)^{\frac{1}{2}} & \left( N - \frac{\beta}{\sqrt{\mu}} \left( 1 + 5 \mu^2 e^{-\mu} \right) \right) \frac{e^{\mu}}{\sqrt{\mu}} < p_k(n) \\
&< \frac{2 \pi}{k}  \left(\frac{k-1}{k-1 + 24n}\right)^{\frac{1}{2}} \frac{e^{\mu}}{\sqrt{\mu}} \left(M + \frac{\beta}{\sqrt{\mu}}\left( 1 + 5 \mu^2 e^{-\mu} \right) \right).
\end{align*}
We assume $a \le b$ and write $b = \lambda a$ for some $\lambda \ge 1$. Then it is sufficient to show
\begin{equation*}\label{lambda}
e^{\mu(a) + \mu(\lambda a) - \mu(\lambda a + a)} > S_{a,k}(\lambda) (k-1 + 24a)^{\frac{3}{4}},
\end{equation*}
where
\begin{equation*}
S_{a,k}(\lambda) :=C_k \frac{ \left(M+ \frac{\beta}{\sqrt{\mu(\lambda a +  a)} }\left( 1 + 5 \mu(\lambda  a+ a)^2 e^{-\mu (\lambda a + a)} \right)  \right) }{\left( N - \frac{\beta}{\sqrt{\mu( \lambda a )}}\left( 1 + 5 \mu(\lambda a)^2 e^{-\mu(\lambda a)} \right) \right) \left(  N - \frac{\beta}{\sqrt{\mu(a)}}  \left( 1 + 5 \mu(a)^2 e^{-\mu (a)} \right)\right) },
\end{equation*}
for $C_k := \frac{k^{\frac{3}{4}}}{2 (\pi(k-1))^{\frac{1}{2}}}$.
For a fixed $a$, the left-hand side of the inequality is increasing for all $\lambda \ge 1$, and the right-hand side is decreasing. Thus, for any given $a$, to prove Theorem \ref{thm:reg-inequality} for $b \ge a$, it suffices to verify the inequality for $\lambda =1$. Taking the natural logarithm of each side, it is straightforward to verify that the inequality holds for $a \ge 1000$ for $k=2,4$, and holds for $a \ge 5 \cdot 10^4$ for $k=3,5$. Then for each of the remaining $a$, we wish to find $\lambda_{a,k}$ such that for $\lambda \ge \lambda_{a,k}$:
\begin{align*}
p_k(a) & \frac{2 \pi}{k} \left( \frac{k-1}{k-1 + 24\lambda a} \right)^{\frac{1}{2}} \left( N - \frac{\beta}{\sqrt{\mu (\lambda a)}} \left( 1 + 5 \mu(\lambda a)^2 e^{-\mu (\lambda a)} \right) \right) \frac{e^{\mu (\lambda a)}}{\sqrt{\mu( \lambda a)}}  > \\
& \frac{2 \pi}{k}  \left(\frac{k-1}{k-1 + 24(\lambda a + a)}\right)^{\frac{1}{2}} \frac{e^{\mu (\lambda a + a)}}{\sqrt{\mu (\lambda a+ a)}} \left(M + \frac{\beta}{\sqrt{\mu (\lambda a + a)}}\left( 1 + 5 \mu( \lambda a + a)^2 e^{-\mu (\lambda a + a)} \right) \right) .
\end{align*}
For $a \ge 20$, $k = 2,4$, $\lambda_{a,k} = \frac{1000}{a}$ suffices. For $a \le 20$, $k = 3,5$, $\lambda_{a,k} = \frac{50000}{a}$ suffices. For smaller $a$, the needed $a \lambda_{a,k} $ values can be as large as $4 \cdot 10^5$, except when $k=5$ and $a=2$, where the larger bound in Theorem \ref{thm:Asums} for $k=5$ causes the needed $\lambda_{a,k}$ values to be much larger. All other cases are reduced to checking a large but finite number of pairs $(a,b)$, where $a \le 5\cdot 10^4$ and $b \le \lambda_{a,k} a$. We carried out these calculations using Sage mathematical software \cite{Sage}. To ease our calculation,
we proved the inequality $ p_5(2) \cdot p_5(b) > p_5(b + 2)$  for $b \ge 75$ with a combinatorial argument (see the end of the section), and used Sage \cite{Sage} to check the remaining pairs.

Now we handle the $k=6$ case. This case is very similar to the cases for
$2 \le k \le 5$, but because of the second summation in (\ref{eq:hagis2}), we have additional, non-dominant terms in our expressions. Using Theorem \ref{thm:Lehmer} and factoring out the leading term, we obtain
\begin{align*}
\frac{\pi}{3} \frac{\sqrt{5}}{\sqrt{24n+5}} \frac{e^{\mu}}{\sqrt{\mu_6}} &\left( N \left(1 -  \eta(n) \right) - \frac{\beta}{\sqrt{\mu}} \left( 1 + \delta(n) \right) \right)
< p_6(n) \\
&< \frac{\pi}{3} \frac{\sqrt{5}}{\sqrt{24n+5}} \frac{e^{\mu}}{\sqrt{\mu_6}} \left( M \left(1 + \eta(n) \right) + \frac{\beta}{\sqrt{\mu}} \left( 1 + \delta(n) \right) \right),
\end{align*}
where $\eta(n) := \left( \frac{2}{5}\right)^{\frac{1}{4}} e^{ \mu \left( 10^{- \frac{1}{2}} - 1 \right) }$.
The desired inequality is implied by
\begin{equation*}
e^{\mu(a) + \mu(\lambda a) - \mu(\lambda a + a)} > S_{a,k}(\lambda) (k-1 + 24a)^{\frac{3}{4}},
\end{equation*}
where
\begin{equation*}
S_{a}(\lambda) =  C_6 \frac{ \left( M( 1 + \eta(\lambda a + a) ) +\frac{\beta}{ \sqrt{\mu((\lambda +1)a)}} \left(1  +  \delta(\lambda a + a) \right)  \right)}{\left(N ( 1 - \eta(a))  - \frac{\beta}{\sqrt{\mu(\lambda a)}}\left(1 +\delta (\lambda a) \right) \right) \left( N( 1 - \eta(a)) - \frac{\beta}{2 \sqrt{\mu_6(a)}}\left(1 +\delta(a) \right) \right) },
\end{equation*}
and $C_6 =  \frac{3 }{ \pi \sqrt{5}} (\frac{6^{\frac{3}{2}}}{\sqrt{5}\pi})^{\frac{1}{2}}$. As before, it suffices to verify that this is true for $\lambda  = 1$, which is straightforward for $a \ge 3500$.
Then for each $a \le 3500$, we wish to find $\lambda_{a,6}$ such that for  all $\lambda \ge \lambda_{a,6}$,
\begin{align*}
p_6(a) &\frac{\pi}{3} \frac{\sqrt{5}}{\sqrt{24(\lambda a)+5}} \frac{e^{\mu(\lambda a)}}{\sqrt{\mu(\lambda a)}} \left(  N( 1 - \eta(\lambda a)) - \frac{\beta}{\sqrt{\mu (\lambda a) }}\left(1  + \delta(\lambda a) \right) \right) > \\
&\frac{\pi}{3} \frac{\sqrt{5}}{\sqrt{24(\lambda a +a)+5}} \frac{e^{\mu(\lambda a + a)}}{\sqrt{\mu ((\lambda + 1)a)}}  \left( M(1 + \eta(\lambda a + 1))+ \frac{\beta}{\sqrt{\mu(\lambda a + a)}} \left( 1 + \delta(\lambda a + a) \right) \right) .
\end{align*}
 It is straightforward to verify that the inequality holds for $\lambda \ge \frac{3500}{a}$ for all $a \ge 4$. For $a = 2,3,4$, the inequality holds for $\lambda \ge \frac{50000}{a}$. This reduces the $k=6$ case to a finite number of pairs $(a,b)$ to check, which we computed with Sage \cite{Sage}.

Finally, we prove that for $b \ge 75$, we have $p_5(b + 2) < 2 \, p_5(b)$.
To do this, we separate the 5-regular partitions of $b+2$ into two disjoint sets.
Let $S_1$ be the set of 5-regular partitions of $b+2$ which contain 1 as a part with multiplicity at least two.
Let $S_2$ contain all the other 5-regular partitions of $b + 2$.
Let $S$ be the set of 5-regular partitions of $b$.
We map $S_1$ and $S_2$ each injectively into $S$.
To map $S_1$ injectively into $S$, for each partition in $S_1$, simply remove two parts~1.

Next, we define an injective map from $S_2$ into $S$.
Let $\ga=(\ga_1,\ga_2,\ldots,\ga_\ell)$ be a partition in~$S_2$.
If $\ga_\ell \ge 2$ and $\ga_1 \ge 7$, then $\ga$ is mapped to $(\ga_2, \ldots,\ga_\ell,1^{\ga_1-2})$
(here, we use exponential notation for multiplicities).
If $\ga_\ell \ge 2$ and $\ga_1 < 7$, then if 2 has multiplicity at least 5 in $\ga$,
replace five parts~2 with eight parts~1.
Otherwise, if $\ga$ has five parts 3, we replace them with thirteen parts~1.
If $\ga$ has five parts 4, then we replace them with eighteen parts~1.
Otherwise, $\ga$  must have at least five parts 6, which we replace with 28 parts~1.
Finally, assume $\ga_\ell = 1$.
If $\ga_{\ell-1} \equiv 1 \pmod{5}$, then we map $\ga$ to
$(\ga_1, \ldots,\ga_{\ell-2},\ga_{\ell-1}-4,1^3)$.
Otherwise,  $\ga$ is mapped to $(\ga_1, \ldots,\ga_{\ell-2},\ga_{\ell-1}-1)$.
Note that the mapping from $S_2$ to $S$
is not onto by considering any 5-regular partition of $b$ which contains exactly two ones.
Thus we obtain the inequality $p_5(b + 2) < 2 p_5(b)$ for $b \ge 75$.

This completes the proof of the inequality stated in Theorem~\ref{thm:reg-inequality}.

The exceptional pairs given in the table are then easily obtained by direct computations.
\qed

\bigskip

\section{The maximum property}\label{sec:max}

We first recall \cite[Theorem 1.1]{BesOno}.
\begin{theorem}\label{thm:BO-max}
Let $n\in \N$. For $n\geq 4$ and $n\neq 7$,
the maximal value $\maxp(n)$
of the partition function on $P(n)$
is attained exactly at the partitions (in exponential notation)
$$
\begin{array}{ll}
(4^{\frac n4 }) & \text{when } n \equiv 0 \pmod 4\\
(5,4^{\frac{n-5}4}) & \text{when } n \equiv 1 \pmod 4\\
(6,4^{\frac{n-6}4}) & \text{when } n \equiv 2 \pmod 4\\
(6,5,4^{\frac{n-11}4}) & \text{when } n \equiv 3 \pmod{4}
\end{array}
$$
For $n=7$, the maximal value is $\maxp(7)=15$,
attained at the two partitions $(7)$ and $(4,3)$.

In particular, if $n\geq 8$, then
$$
\maxp(n)=\begin{cases}
5^{\frac{n}{4}}\ \ \ \ \ \ \ \
 & \text {\rm if $n\equiv 0\pmod 4$},\\
7\cdot 5^{\frac{n-5}{4}}\ \ \ \ \ \ \ \ &\text {\rm if $n\equiv 1\pmod 4$},\\
11\cdot 5^{\frac{n-6}{4}}\ \ \ \ \ \ \ \ &\text {\rm if $n\equiv 2\pmod 4$},\\
11\cdot7\cdot 5^{\frac{n-11}{4}}\ \ \ \ \ \ \ \ &\text {\rm if $n\equiv 3\pmod 4$}.
\end{cases}
$$
\end{theorem}

Since the partitions where the maximum of $p(n)$ is attained
on $P(n)$ are $k$-regular for any $k>6$, in the following
it suffices to consider the cases $k\in \{2,3,4,5,6\}$.

\begin{theorem}\label{thm:max-regular}
Let $k\in \N$, $k>1$. Let $n\in \N$.

\begin{enumerate}
\item[(i)] $k=2$.
For $n\geq 9$ and $n\neq 11$,
the maximal value $\maxp_2(n)$
of $p_2(n)$ on $P_2(n)$
is attained exactly at the partitions
$$
\begin{array}{ll}
(9^a,3^b) & \text{when } n \equiv 0 \pmod 3 \\
(9^a,7,3^b) & \text{when } n \equiv 1 \pmod 3 \\
(9^a,7,7,3^b) & \text{when } n \equiv 2 \pmod 3 \\
\end{array}
$$
where $a,b\in \N_0$ may be chosen arbitrarily
as long as we have partitions of~$n$.
\\
In particular, we have
$$
\maxp_2(n)= \left\{
\begin{array}{ll}
2^{\frac n3} & \text{when } n \equiv 0 \pmod 3\\
5\cdot 2^{\frac {n-7}3} & \text{when } n \equiv 1 \pmod 3\\
5^2\cdot 2^{\frac {n-14}3} & \text{when } n \equiv 2 \pmod 3\\
\end{array}
\right.
$$
\item[(ii)] $k=3$.
For $n\geq 2$ and $n\neq 3$,
the maximal value $\maxp_3(n)$
of $p_3(n)$ on $P_3(n)$
is attained exactly at the partitions
$$
\begin{array}{ll}
(4^a,2^b) & \text{when } n \equiv 0 \pmod 2\\
(5,4^a,2^b) & \text{when } n \equiv 1 \pmod 2\\
\end{array}
$$
where $a,b\in \N_0$ may be chosen arbitrarily
as long as we have partitions of~$n$.
\\
In particular, we have
$$
\maxp_3(n)= \left\{
\begin{array}{ll}
2^{\frac n2} & \text{when } n \equiv 0 \pmod 2\\
5\cdot 2^{\frac {n-5}2} & \text{when } n \equiv 1 \pmod 2\\
\end{array}
\right.
$$
\item[(iii)] $k=4$.
For $n\geq 2$,
the maximal value $\maxp_4(n)$
of $p_4(n)$ on $P_4(n)$
is attained exactly at the partitions
$$
\begin{array}{ll}
(6^a,3^b) & \text{when } n \equiv 0 \pmod 3\\
(6^a,3^b,2,2),(7,6^a,3^b),(6^a,5,3^b,2),
(6^a,5,5,3^b) & \text{when } n \equiv 1 \pmod 3\\
(6^a,3^b,2) , (6^a,5,3^b)& \text{when } n \equiv 2 \pmod 3\\
\end{array}
$$
where $a,b\in \N_0$ may be chosen arbitrarily
as long as we have partitions of~$n$,
and with the understanding that
partitions with given parts $2,5,7$ of positive multiplicity
do not occur when $n$ is too small.
\\
In particular, we have
$$
\maxp_4(n)= \left\{
\begin{array}{ll}
3^{\frac n3} & \text{when } n \equiv 0 \pmod 3\\
4\cdot 3^{\frac {n-4}3} & \text{when } n \equiv 1 \pmod 3\\
2\cdot 3^{\frac {n-2}3} & \text{when } n \equiv 2 \pmod 3\\
\end{array}
\right.
$$
\item[(iv)] $k=5$.
For $n\geq 2$,
the maximal value $\maxp_5(n)$
of $p_5(n)$ on $P_5(n)$
is attained exactly at the partitions
$$
\begin{array}{ll}
(4^{\frac n4 }) & \text{when } n \equiv 0 \pmod 4\\
(4^{\frac{n-5}4},3,2),(6,4^{\frac{n-9}4},3) & \text{when } n \equiv 1 \pmod 4\\
(4^{\frac{n-2}4},2),(6,4^{\frac{n-6}4})  & \text{when } n \equiv 2 \pmod 4\\
(4^{\frac{n-3}4},3)  & \text{when } n \equiv 3 \pmod 4\\
\end{array}
$$
with the understanding that
partitions with given parts $2,3,6$ of positive multiplicity
do not occur when $n$ is too small.
\\
In particular, we have
$$
\maxp_5(n)= \left\{
\begin{array}{ll}
5^{\frac n4} & \text{when } n \equiv 0 \pmod 4\\
6\cdot 5^{\frac {n-5}4} & \text{when } n \equiv 1 \pmod 4\\
2\cdot 5^{\frac {n-2}4} & \text{when } n \equiv 2 \pmod 4\\
3\cdot 5^{\frac {n-3}4} & \text{when } n \equiv 3 \pmod 4\\
\end{array}
\right.
$$
\item[(v)] $k=6$.
For $n\geq 2$,
the maximal value $\maxp_6(n)$
of $p_6(n)$ on $P_6(n)$
is attained exactly at the partitions
$$
\begin{array}{ll}
(4^{\frac n4 }) & \text{when } n \equiv 0 \pmod 4\\
(5,4^{\frac{n-5}4}) & \text{when } n \equiv 1 \pmod 4\\
(4^{\frac{n-2}4},2)  & \text{when } n \equiv 2 \pmod 4\\
(4^{\frac{n-3}4},3)  & \text{when } n \equiv 3 \pmod 4\\
\end{array}
$$
\\
In particular, we have
$$
\maxp_6(n)= \left\{
\begin{array}{ll}
5^{\frac n4} & \text{when } n \equiv 0 \pmod 4\\
7\cdot 5^{\frac {n-5}4} & \text{when } n \equiv 1 \pmod 4\\
2\cdot 5^{\frac {n-2}4} & \text{when } n \equiv 2 \pmod 4\\
3\cdot 5^{\frac {n-3}4} & \text{when } n \equiv 3 \pmod 4\\
\end{array}
\right.
$$
\end{enumerate}
\end{theorem}

\begin{proof}
(i) We will need the partitions where $\maxp_2(n)$ is attained
for $n\leq 22$;
these are given in Table~\ref{table:4} (computed by Maple).
We see that the assertion holds as stated up to $n=22$.
\medskip

\begin{table}[h]
\caption{\label{table:4} Maximum value partitions $\mu$ for $k=2$}
\begin{tabular}{|r|cc|cc|cc| }
\hline \rule[-3mm]{0mm}{8mm}
$n$ && $p_2(n)$ && $\maxp_2(n)$ && $\mu$ \\
\hline
$1$  && 1 && 1 && (1) \\
$2$  && 1 && 1 && (1,1) \\
$3$  && 2 && 2 && (3) \\
$4$  && 2 && 2 && (3,1) \\
$5$ && 3 && 3 && (5) \\
$6$&&  4 && 4 && $(3^2)$ \\
$7$&& 5 && 5 && (7)\\
$8$&& 6 && 6 && (5,3) \\
$9$&&  8 && 8 && $(9), (3^3)$ \\
$10$&& 10  && 10 && (7,3) \\
$11$&& 12 && 12 && $(11), (5,3^2)$ \\
$12$&& 15 &&  16 && $(9,3),(3^4)$ \\
$13$&& 18 &&  20 && $(7,3^2)$ \\
$14$&& 22 && 25 && $(7^2)$ \\
$15$&& 27 && 32 && $(9,3^2), (3^5)$ \\
$16$&& 32 && 40 && $(9,7), (7,3^3)$\\
$17$&& 38 && 50 && $(7^2,3)$\\
$18$&& 46 && 64 && $(9^2), (9,3^3),(3^6)$\\
$19$&& 54 && 80 && $(9,7,3),(7,3^4)$ \\
$20$&& 64 && 100&& $(7^2,3^2)$ \\
$21$&& 76 && 128&& $(9^2,3), (9,3^4), (3^7)$\\
$22$&& 89 && 160&& $(9,7,3^2), (7,3^5)$\\
\hline
\end{tabular}
\end{table}
\medskip

Now take $n>22$.
Let $\mu \in P_2(n)$ be such that $p_2(\mu)$ is maximal;
let $m_j$ be the multiplicity of a part $j$ in $\mu$.
Suppose $\mu$ has a part $j=2h+1\geq 19$;
let $\{h,h+1\}=\{2l,h'\}$.
Then by Theorem~\ref{thm:reg-inequality} and Table~\ref{table:4},
replacing $j$ by the parts $h',2l-3,3$ in $\mu$
would produce a partition $\nu\in P_2(n)$
such that $p_2(\nu)>p_2(\mu)$.
Thus  $\mu$ has no parts $j\geq 19$.
By Table~\ref{table:4},  a part $j\in \{13,15,17\}$
could be replaced in $\mu$ by a partition in $P_2(j)$
giving a partition $\nu\in P_2(n)$ of larger
$p_2$-value.
Thus $\mu$ only has odd parts $j\leq 11$.

Any two parts $(11^2)$, $(11,9)$, $(11,7)$, $(11,5)$, $(11,3)$, $(11,1)$
can be replaced by a $2$-regular partition to obtain a higher $p_2$-value,
see Table~\ref{table:4};
thus $m_{11}=0$.
Also $(7^3)$, $(7,5)$, $(7,1)$ can be replaced to obtain a higher
$p_2$-value. Thus $m_7\le 2$, and the part 7 can only occur when
$\mu$ is of the form $(9^a,7,3^b)$ or $(9^a,7^2,3^b)$;
in the first case  $n\equiv 1 \mod 3$, in the second case we have $n\equiv 2 \mod 3$.
Also $(5^2)$ can be replaced by $(7,3)$ to obtain a higher $p_2$-value,
so $m_5\le 1$; then replacing $(9,5)$ or $(5,3^3)$ by $(7^2)$,
and $(5,1)$ by $(3^2)$ shows that $\mu$ has no part 5. Hence if $\mu$ has no part 7, then $\mu$ is of the form $(9^a,3^b)$,
and $n\equiv 0 \mod 3$.
As $p_2((9))=p_2((3^3))$, the part 9 and the parts $3,3,3$ can always be used interchangeably.
Now for $n\ge 19$ and any congruence $n\equiv c \mod 3$,
$c\in \{0,1,2\}$,
we have found just one type of 2-regular partition
maximizing the $p_2$-value, namely $(9^a,7^c,3^b)$,
with $a,b\in \N_0$ such that $(3a+b)\cdot 3+c\cdot 7=n$,
where $p_2((9^a,7^c,3^b))=2^{3a+b}5^c=\maxp_2(n)$.
This proves the claim for $k=2$.
\medskip

(ii) By Table~\ref{table:5} the claim holds for $n\le 16$.
So we assume now that $n>16$.
\begin{table}[h]
\caption{\label{table:5} Maximum value partitions $\mu$ for $k=3$}
\begin{tabular}{|r|cc|cc|cc| }
\hline \rule[-3mm]{0mm}{8mm}
$n$ && $p_3(n)$ && $\maxp_3(n)$ && $\mu$ \\
\hline
$1$  && 1 && 1 && (1) \\
$2$  && 2 && 2 && (2) \\
$3$  && 2 && 2 && (2,1) \\
$4$  && 4 && 4 && $(4),(2^2)$ \\
$5$ && 5 && 5 && (5) \\
$6$&&  7 && 8 && $(4,2),(2^3)$ \\
$7$&& 9 && 10 && (5,2)\\
$8$&&  13 && 16 && $(4^2),(4,2^2),(2^4)$ \\
$9$&&  16 && 20 && $(5,4), (5,2^2)$ \\
$10$&& 22  && 32 && $(4^2,2),(4,2^3),(2^5)$ \\
$11$&& 27 && 40 && $(5,4,2),(5,2^3)$ \\
$12$&& 36 &&  64 && $(4^3),(4^2,2^3),(4,2^4),(2^6)$ \\
$13$&& 44 &&  80 && $(5,4^2),(5,4,2^2),(5,2^4)$ \\
$14$&& 57 && 128 && $(4^3,2),(4^2,2^3),(4,2^5),(2^7)$ \\
$15$&& 70 && 160&& $(5,4^2,2),(5,4,2^3),(5,2^5)$ \\
$16$&& 89 && 256&& $(4^4),(4^3,2^2),(4^2,2^4),(4,2^6),(2^8)$ \\
\hline
\end{tabular}
\end{table}

Let $\mu \in P_3(n)$
be such that $p_3(\mu)$ is maximal.
Suppose $\mu$ has a part $j\geq 17$.
Replace $j$ by $\nu_j=(j-2,2)$ if $j\equiv 1 \mod 3$, and by $\nu_j=(j-4,4)$
if $j\equiv 2 \mod 3$.
By Theorem~\ref{thm:reg-inequality} we have
$p_3(j) < p_3(\nu_j)$.
Thus $\mu$ only has parts $\le 16$.
By Table~\ref{table:5}, any of these can be replaced by a partition
of the form $(5^a,4^b,2^c,1^d)$ to increase the $p_3$-value,
and we note that the parts 4 and $2,2$ can be used interchangeably.
Hence only parts $1,2,4,5$ can appear in~$\mu$.
By Table~\ref{table:5},  the following replacements would increase
the $p_3$-value:
$(5^2) \to (2^5)$, $(5,1)\to (2^3)$, $(4,1),(2^2,1)\to (5)$,
$(1^2)\to (2)$.
This implies that $\mu$ can only have one of the forms
$(4^a,2^b)$ or $(5,4^a,2^b)$, where in the first case $n\equiv 0 \mod 2$, in the second case
$n \equiv 1 \mod 2$.
Hence $\maxp_3(n)=2^{\frac n2}$ when $n$ is even, and
$\maxp_3(n)=5\cdot 2^{\frac {n-5}2}$ when $n$ is odd.

\medskip

(iii)
By Table~\ref{table:6} the claim holds for $n\le 15$,
so now take $n\ge 16$.
Let
$\mu \in P_4(n)$ be such that $p_4(\mu)$ is maximal.
Note that by Table~\ref{table:6}, we may always exchange a part
$6$ against the parts $3,3$ without
changing  the $p_4$-value.
Suppose $\mu$ has a part $j\geq 9$.
Replace $j$ by $\nu_j=(j-2,2)$, when $j\not\equiv 2 \mod 4$, or by $\nu_j=(j-3,3)$
when $j\equiv 2 \mod 4$.
By Theorem~\ref{thm:reg-inequality}, $p_4(j)<p_4(\nu_j)$; hence $\mu$ only has parts $\le 7$.
Replacing $(7^2)$ by $(6^2,2)$, $(7,5)$ by $(6^2)$, $(7,2)$ by $(6,3)$,
$(7,1)$ by $(6,2)$ shows that $\mu$ can have a part $7$ only when it is
of the form $(7,6^a,3^b)$, and then $n\equiv 1 \mod 3$.
By Table~\ref{table:6}, in these partitions we may exchange $7$ with $(5,2)$ or $(3,2^2)$,  and $(7,3)$ with $(5^2)$ without changing the $p_4$-value.

Now assume that $\mu$ has no part~7.
Replacing $(5^3)$ by $(6^2,3)$, $(5^2,2)$ by $(6^2)$, $(5,1)$ by $6$,
shows that $\mu$ can have a part $5$ only when  $n\equiv 1 \mod 3$
and it is of the form $(6^a,5,3^b,2)$ or $(6^a,5^2,3^b)$ already discussed above,
or  $n\equiv 2 \mod 3$ and it is of the form $(6^a,5,3^b)$.
Note that $5$ can be exchanged with $(3,2)$ without changing the $p_4$-value.

Finally, when $\mu$ has no parts 5 and 7, the replacements of
$(6,1)$ by $7$, $(2^3)$ by~6,
$(3,1)$ by $(2^2)$,
$(2,1)$ by~$3$, $(1^2)$ by~$2$
show that $\mu$ can have no part~1 and $m_2\le 2$.
Then $\mu$ has one of the forms $(6^a,3^b)$, $(6^a,3^b,2)$ or $(6^a,3^b,2^2)$,
when $n$ is congruent to $0,2,1 \mod 3$, respectively.

Together with the remarks above, we then have
$\maxp_4(n)=3^{\frac n3}$ when $n\equiv 0 \mod 3$,
$\maxp_4(n)=4\cdot 3^{\frac {n-4}3}$ when  $n\equiv 1 \mod 3$,
and
$\maxp_4(n)=2\cdot 3^{\frac {n-2}3}$ when  $n\equiv 2 \mod 3$,
attained at the partitions as stated in the claim for $k=4$.

\begin{table}[h]
\caption{\label{table:6} Maximum value partitions $\mu$ for $k=4$}
\begin{tabular}{|r|cc|cc|cc| }
\hline \rule[-3mm]{0mm}{8mm}
$n$ && $p_4(n)$ && $\maxp_4(n)$ && $\mu$ \\
\hline
$1$  && 1 && 1 && (1) \\
$2$  && 2 && 2 && (2) \\
$3$  && 3 && 3 && (3) \\
$4$  && 4 && 4 && (2,2) \\
$5$ && 6 && 6 && (5), (3,2) \\
$6$&&  9 && 9 && $(6),(3^2)$ \\
$7$&& 12 && 12 && $(7),(5,2),(3,2^2)$\\
$8$&&  16 && 18 && $(6,2),(5,3),(3^2,2)$ \\
$9$&&  22 && 27 && $(6,3), (3^3)$ \\
$10$&& 29  && 36 && $(7,3),(6,2^2),(5^2),(5,3,2)(3^2,2^2)$ \\
$11$&& 38 && 54 && $(6,5),(6,3,2), (5,3^2), (3^3,2)$\\
$12$ && 50 && 81 && $(6^2), (6,3^2), (3^4)$\\
$13$ && 64 && 108 && $(7,6), (7,3^2),(6,5,2), (6,3,2^2),(5^2,3),
(5,3^2,2),(3^3,2^2)$\\
$14$ && 82 && 162 && $(6^2,2),(5^2,3),(6,3^2,2),(5,3^3),(3^4,2)$\\
$15$ && 105 && 243 && $(6^2,3),(6,3^3),(3^5)$\\
\hline
\end{tabular}
\end{table}
\medskip

(iv)
Table~\ref{table:7} shows that the assertion is true for $n\le 12$.
Take $n\ge 13$, and
let $\mu \in P_5(n)$ be such that $p_5(\mu)$ is maximal.
Note that by Table~\ref{table:7} we may always exchange a part
$6$ against the parts $4,2$ without
changing  the $p_5$-value.
Suppose $\mu$ has a part $j\geq 7$.
Replace $j$ by $\nu_j=(j-3,3)$, when $j\not\equiv 3 \mod 5$, or by $\nu_j=(j-4,4)$
when $j\equiv 3 \mod 5$.
By Table~\ref{table:7} and Theorem~\ref{thm:reg-inequality}
$p_5(j)<p_5(\nu_j)$; hence $\mu$ only has parts $\le 6$.

Replacing $(6^2)$ by $(4^3)$, $(6,2)$ by $(4^2)$, $(6,1)$ by $(4,3)$,
$(3^2)$ by $6$, $(3,1)$ or $(2^2)$ by 4, $(2,1)$ by 3 and $(1^2)$ by 2
increases the $p_5$-value.
Hence $\mu$ can only have the forms stated in (iv), and the
assertion about the $\maxp_5$-value also follows.

\begin{table}[h]
\caption{\label{table:7} Maximum value partitions $\mu$ for $k=5$}
\begin{tabular}{|r|cc|cc|cc| }
\hline \rule[-3mm]{0mm}{8mm}
$n$ && $p_5(n)$ && $\maxp_5(n)$ && $\mu$ \\
\hline
$1$  && 1 && 1 && (1) \\
$2$  && 2 && 2 && (2) \\
$3$  && 3 && 3 && (3) \\
$4$  && 5 && 5 && (4) \\
$5$ && 6 && 6 && (3,2) \\
$6$&&  10 && 10 && (6), (4,2) \\
$7$&& 13 && 15 && (4,3)\\
$8$&&  19 && 25 && $(4^2)$ \\
$9$&&  25 && 30 && (6,3), (4,3,2) \\
$10$&& 34  && 50 && $(6,4), (4^2,2)$ \\
$11$&& 44  && 75 && $(4^2,3)$ \\
$12$&& 60  &&  125 && $(4^3)$ \\
\hline
\end{tabular}
\end{table}
\medskip

(v)
Table~\ref{table:8} shows that the assertion is true for $n\le 10$.
Let $n\ge 11$, and
let $\mu \in P_6(n)$ be such that $p_6(\mu)$ is maximal.
Suppose $\mu$ has a part $j\geq 7$.
Replace $j$ by $\nu_j=(j-3,3)$, when $j\equiv 4 \mod 6$, or by $\nu_j=(j-4,4)$
when $j\not\equiv 4 \mod 6$.
By Table~\ref{table:8} and Theorem~\ref{thm:reg-inequality}
$p_6(j)<p_6(\nu_j)$; hence $\mu$ only has parts $\le 5$.
Replacing $(5^2)$ by $(4^2,2)$, $(5,1)$ by $(4,2)$, $(5,2)$ by $(4,3)$,
$(5,3)$ by $(4^2)$, $(3^2)$ by $(4,2)$, $(3,2)$ by 5,
$(3,1)$ or $(2^2)$ by 4, $(2,1)$ by 3 and $(1^2)$ by 2
increases the $p_6$-value.
Hence $\mu$ can only have the forms stated in (v), and the
assertion about the $\maxp_6$-value also follows in this final case.

\begin{table}[h]
\caption{\label{table:8} Maximum value partitions $\mu$ for $k=6$}
\begin{tabular}{|r|cc|cc|cc| }
\hline \rule[-3mm]{0mm}{8mm}
$n$ && $p_6(n)$ && $\maxp_6(n)$ && $\mu$ \\
\hline
$1$  && 1 && 1 && (1) \\
$2$  && 2 && 2 && (2) \\
$3$  && 3 && 3 && (3) \\
$4$  && 5 && 5 && (4) \\
$5$ && 7 && 7 && (5) \\
$6$&&  10 && 10 && (4,2) \\
$7$&& 14 && 15 && (4,3)\\
$8$&&  20 && 25 && $(4^2)$ \\
$9$&&  27 && 35 && (5,4) \\
$10$&& 37  && 50 && $(4^2,2)$ \\
\hline
\end{tabular}
\end{table}
\end{proof}

\section{Concluding remarks}

We note that recently also other multiplicative properties of the partition function have been studied
and one might ask whether those also hold for the generating function for $k$-regular partitions.
Originating in a conjecture by William Chen, DeSalvo and Pak in \cite{DeSalvoPak} have proved
log-concavity for the partition function for all $n>25$;
do we have an analogue of this?

Indeed, there is computational evidence for a version of Chen's conjecture
for $k$-regular partitions, i.e.,
when $n>n_0$ (with $n_0$ being relatively small) then for all
$m\in \{2,3, \ldots, n-1\}$:
$$
p_k(n)^2 > p_k(n-m)p_k(n+m)\:.
$$

The inequality
$p_k(1)p_k(n) = p_k(n) < p_k(n+1)$
has an easy combinatorial proof by an injection $P_k(n) \to P_k(n+1)$.
One may ask whether there is also a combinatorial argument for
proving the inequality in Theorem~\ref{thm:reg-inequality}.

As mentioned before,
the number $p_k(n)$ is equal to the number of partitions where no part
has a multiplicity $\geq k$.
But when we extend the corresponding (same)
generating function $p_k(n)$ to the set of partitions with all multiplicities being $<k$
in analogy to the extension to the set $P_k(n)$,
the behavior is quite different.
In particular, the maximal values on the
two different partition sets to a given $n\in\N$ are in general different,
and for the second extension, the sets of partitions giving the maximal value
are more complicated.

Hagis' formulae play a crucial role in this paper;  
as pointed out by the referee, 
results of this type have been obtained recently 
in a much wider context.  
Indeed, Bringmann and Ono \cite{BringmannOno} 
give exact formulae for the coefficients 
of all weight 0 modular functions 
and also all harmonic Maass forms of non-positive weight. 
This work might be employed to study other 
partition-related functions 
defined similarly as our maxp-functions.

\bigskip

{\bf Acknowledgement.}
The authors thank Michael Griffin for assisting with the calculations at the end of Section~\ref{sec:proof-reg-inequality}.


\begin{thebibliography}{GKZ}
\bibitem{Abram} M.\  Abramovitz, I.\ Stegun, \emph{Handbook of Mathematical Functions with Formulas, Graphs, and Mathematical Tables}, Courier Dover Publications, 1972, 378.

\bibitem{BesOno} C.\ Bessenrodt and K.\ Ono,
\emph{Maximal multiplicative properties of  partitions},
to appear in: Annals of Comb. (arXiv:1403.3352)

\bibitem{BringmannOno} K.\ Bringmann and K.\ Ono,
\emph{Coefficients of harmonic Maass forms},
preprint 2014 

\bibitem{BruinierOno} J.\ H.\ Bruinier and K.\ Ono, \emph{Algebraic formulas for the coefficients of half-integral weight harmonic weak Maass forms},
Adv. in Math., \textbf{246} (2013), 198--219.

\bibitem{DeSalvoPak} S.\ DeSalvo and I.\ Pak, \textit{Log-concavity of the partition function},
to appear in: Ramanujan Journal. (arXiv:1310.7982)

\bibitem{Hagis} P.\ Hagis, \textit{Partitions with a Restriction on the Multiplicity of the Summands}, Trans. Amer. Math. Soc. \textbf{155} (1971), No. 2, 375-384.

\bibitem{Hardy} G.\ H.\ Hardy, E.\ M.\ Wright, \textit{An introduction to the theory of numbers}, Oxford Press, Sixth Edition (2008).

\bibitem{Iwaniec} H.\ Iwaniec, \textit{Topics in Classical Automorphic Forms}, Amer. Math. Soc. (1997).

\bibitem{Lehmer} D.\ Lehmer, \textit{On the series for the partition function}, Trans. Amer. Math. Soc. \textbf{43} (1938), No. 2, 271-295.

\bibitem{NicRob} J.\ Nicolas, G.\ Robin \textit{Majorations explicites pour le nombre de diviseurs de n} Canad. Math. Bull. \textbf{39} (1983), 485-492.

\bibitem{Salie} H.\ Salie, \textit{Zur Absch\"atzung der Fourierkoeffizienten ganzer Modulformen}, Math. Z. \textbf{36} (1933), No. 1, 263-278.

\newcommand{\etalchar}[1]{$^{#1}$}
\bibitem[S{\etalchar{+}}09]{Sage}
W.\thinspace{}A.\ Stein et~al., \emph{{S}age {M}athematics {S}oftware ({V}ersion
  6.1.1)}, The Sage Development Team, 2014, {\tt http://www.sagemath.org}.

\bibitem{Wigert} S. \ Wigert, \textit{Sur quelques fonctions arithm\'etiques}, Acta Math. \textbf{37} (1914), pp. 113-140.

\end{thebibliography}
\end{document}